\documentclass[global]{article}%
\usepackage{amsfonts}
\usepackage{amsmath}
\usepackage{amssymb}
\usepackage{graphicx}%
\newtheorem{theorem}{Theorem}

\newtheorem{corollary}[theorem]{Corollary}

\newtheorem{remark}[theorem]{Remark}

\newenvironment{proof}[1][Proof]{\noindent\textbf{#1.} }{\ \rule{0.5em}{0.5em}}
\begin{document}

\title{Korovkin type theorem for iterates of certain positive linear operators}
\author{N. I. Mahmudov\\Department of Mathematics\\Eastern Mediterranean University \\Gazimagusa, TRNC via Mersin 10, Turkey \\email: nazim.mahmudov@emu.edu.tr}
\date{Submitted: March 15, 2011}
\maketitle

\begin{abstract}
In this paper we prove that if $T:C\left[  0,1\right]  \rightarrow C\left[
0,1\right]  $ is a positive linear operator with $T\left(  e_{0}\right)
=1\ $and$\ T\left(  e_{1}\right)  -e_{1}$ does not change the sign, then the
iterates $T^{m}$ converges to some positive linear operator $T^{\infty}$
$:C\left[  0,1\right]  \rightarrow C\left[  0,1\right]  $ and we derive
quantitative estimates in terms of modulii of smoothness. This result enlarges
the class of operators for which the limit of the iterates can be computed and
the quantitative estimates of iterates can be given.

\end{abstract}

\section{Introduction and Main results}

The methods employed to study the convergence of iterates of some operators
include Matrix Theory methods, like stochastic matrices, Korovkin-type
theorems, quantitative results about the approximation of functions by
positive linear operators, point theorems, or methods from the theory of
$C_{0}$-semigroups, like Trotter's approximation theorem, see \cite{Alt1}%
-\cite{wenz}. However, these results fail to calculate the iterate limit and
give the quantitative estimates of iterates for many classical operators. In
respect to this, we mention recent works by I. Gavrea and M. Ivan
\cite{gavrea1}, \cite{gavrea2} and I. Rasa \cite{rasa1}, \cite{rasa2}.

In this paper we establish quantitative Korovkin type theorem for the iterates
of certain positive linear operators $T:C\left[  0,1\right]  \rightarrow
C\left[  0,1\right]  $ satisfying $T\left(  e_{0}\right)  =e_{0},\ \ T\left(
e_{1}\right)  -e_{1}\leq0$ (or $\geq0$). As a consequence of our results, we
obtain the quantitaive estimates for the iterates of almost all classical and
new positive linear operators. Notice that quantitative Korovkin type theorems
for a sequence of positive linear operators are studied in \cite{wang2},
\cite{mah2}.

Let $C\left[  0,1\right]  $ be the set of all real-valued and continuous
functions defined on the compact interval $\left[  0,1\right]  $ endowed with
the $\sup$-norm $\left\Vert f\right\Vert :=\sup\left\{  \left\vert f\left(
x\right)  \right\vert :x\in\left[  0,1\right]  \right\}  $. $W_{2,\infty
}\left[  0,1\right]  $ is defined as follows
\begin{align*}
W_{2,\infty}\left[  0,1\right]   &  :=\left\{  f\in C\left[  0,1\right]
:f^{\prime}\text{ absolutely continuous, }\left\Vert f^{\prime\prime
}\right\Vert _{L_{\infty}}<\infty\right\}  ,\\
\left\Vert f\right\Vert _{L_{\infty}} &  :=\text{vrai}\max\left\{  \left\vert
f^{\prime\prime}\left(  x\right)  \right\vert :0\leq x\leq1\right\}  ,
\end{align*}
and $L_{\infty}$ is the space of essentially bounded measurable functions
endowed with $\left\Vert \cdot\right\Vert _{L_{\infty}}$ norm.

The main tools to measure the degree of convergence of the powers of positive
linear operators are the moduli of smoothness of first and second order. For
$f\in C\left[  0,1\right]  $ and $\delta\geq0$ we have%
\begin{align*}
\omega_{1}\left(  f;\delta\right)   &  :=\sup\left\{  \left\vert f\left(
x+h\right)  -f\left(  x\right)  \right\vert :x,x+h\in\left[  0,1\right]
,0\leq h\leq\delta\right\}  ,\\
\omega_{2}\left(  f;\delta\right)   &  :=\sup\left\{  \left\vert f\left(
x+h\right)  -2f\left(  x\right)  +f\left(  x-h\right)  \right\vert :x,x\pm
h\in\left[  0,1\right]  ,0\leq h\leq\delta\right\}  .
\end{align*}
For $f\in C\left[  0,1\right]  $ we define the extension $f_{h}:\left[
-h,1+h\right]  \rightarrow\mathbb{R},$ with $h>0$, by%
\[
f_{h}\left(  x\right)  :=\left\{
\begin{array}
[c]{c}%
P_{-}\left(  x\right)  ,\ \ \ -h\leq x\leq0,\\
f\left(  x\right)  ,\ \ \ 0\leq x\leq1,\\
P_{+}\left(  x\right)  ,\ \ \ 1\leq x\leq1+h,
\end{array}
\right.
\]
where $P_{-}$ and $P_{+}$ are at most order best approximants to $f$ on the
indicated intervals. Then Zhuk's function $Z_{h}f$ is defined by means of the
second order Steklov means%
\[
Z_{h}f\left(  x\right)  :=\frac{1}{h}\int_{-h}^{h}\left(  1-\frac{\left\vert
t\right\vert }{h}\right)  f_{h}\left(  x+t\right)  dt,\ \ 0\leq x\leq1.
\]
It can be shown that $Z_{h}f\in W_{2,\infty}\left[  0,1\right]  .$ It is known
that for sufficiently large $l$ and a fixed $\varepsilon>0$ we have%
\begin{align}
\left\Vert f-g\right\Vert  &  \leq\left\Vert f-Z_{h}f\right\Vert +\left\Vert
B_{l}\left(  Z_{h}f\right)  -Z_{h}f\right\Vert \leq\frac{3}{4}\omega
_{2}\left(  f;h\right)  +\varepsilon,\nonumber\\
\left\Vert g^{\prime}\right\Vert  &  \leq\left\Vert \left(  Z_{h}f\right)
^{\prime}\right\Vert \leq\frac{1}{h}\left(  2\omega_{1}\left(  f;h\right)
+\frac{3}{2}\omega_{2}\left(  f;h\right)  \right)  ,\nonumber\\
\left\Vert g^{\prime\prime}\right\Vert  &  \leq\left\Vert \left(  Z_{\delta
}f\right)  ^{\prime\prime}\right\Vert _{L_{\infty}}\leq\frac{1}{\delta^{2}%
}\frac{3}{2}\omega_{2}\left(  f;\delta\right)  .\label{zhuk}%
\end{align}

For any positive linear operator $T:C\left[  0,1\right]  \rightarrow C\left[
0,1\right]  $, we define the powers of $T$ by%
\[
T^{0}=I,\ \ \ T^{1}=T,\ \ \ T^{m+1}=T\circ T^{m},\ \ \ m\in\mathbb{N}.
\]

Let $e_{i}:\left[  0,1\right]  \rightarrow R$ be the monomial functions
$e_{i}\left(  x\right)  =x^{i},$ $i=0,1,2.$

Now we formulate the main results of the paper. It shows that under the
conditions $T\left(  e_{0}\right)  =1\ $and $T\left(  e_{1}\right)  -e_{1}$
does not change the sign, the iterates of $T:C\left[  0,1\right]  \rightarrow
C\left[  0,1\right]  $ converges to some linear positive operator $T^{\infty
}:C\left[  0,1\right]  \rightarrow C\left[  0,1\right]  .$

\begin{theorem}
\label{thm0} Suppose that $T:C\left[  0,1\right]  \rightarrow C\left[
0,1\right]  $ is a positive linear operator such that $T\left(  e_{0}\right)
=e_{0}.$

\begin{enumerate}
\item[(a)] If $T\left(  e_{1}\right)  \leq e_{1}$, then there exists a linear
positive positive operator $T^{\infty}:C\left[  0,1\right]  \rightarrow
C\left[  0,1\right]  $ such that
\[
\lim_{n\rightarrow\infty}T^{m}\left(  f\right)  =T^{\infty}\left(  f\right)
,\ \ \ f\in C\left[  0,1\right]  ,
\]
and the following pointwise estimate%
\begin{align}
\left\vert T^{\infty}\left(  f;x\right)  -T^{m}\left(  f;x\right)
\right\vert  &  \leq3\omega_{2}\left(  f;\left\vert \left(  T^{m}-T^{\infty
}\right)  \left(  e_{1};x\right)  \right\vert \right)  +2\omega_{1}\left(
f;\left\vert \left(  T^{m}-T^{\infty}\right)  \left(  e_{1};x\right)
\right\vert \right) \nonumber\\
&  +\frac{3}{4}\omega_{2}\left(  f;\sqrt{\left\vert \left(  T^{\infty}%
-T^{m}\right)  \left(  e_{2};x\right)  \right\vert +2\left\vert \left(
T^{m}-T^{\infty}\right)  \left(  e_{1};x\right)  \right\vert }\right)
\label{ss1}%
\end{align}
holds true \textit{for }$x\in\left[  0,1\right]  $\textit{ and }$f\in C\left[
0,1\right]  .$

\item[(b)] If $T\left(  e_{1}\right)  \geq e_{1}$, then there exists a linear
positive positive operator $T^{\infty}:C\left[  0,1\right]  \rightarrow
C\left[  0,1\right]  $ such that
\[
\lim_{n\rightarrow\infty}T^{m}\left(  f\right)  =T^{\infty}\left(  f\right)
,\ \ \ f\in C\left[  0,1\right]  ,
\]
and the following pointwise estimate%
\begin{align}
\left\vert T^{\infty}\left(  f;x\right)  -T^{m}\left(  f;x\right)
\right\vert  &  \leq3\omega_{2}\left(  f;\left\vert \left(  T^{m}-T^{\infty
}\right)  \left(  e_{1};x\right)  \right\vert \right)  +2\omega_{1}\left(
f;\left\vert \left(  T^{m}-T^{\infty}\right)  \left(  e_{1};x\right)
\right\vert \right) \nonumber\\
&  +\frac{3}{4}\omega_{2}\left(  f;\sqrt{\left\vert \left(  T^{\infty}%
-T^{m}\right)  \left(  e_{2};x\right)  \right\vert }\right)  \label{sss1}%
\end{align}
holds true \textit{for }$x\in\left[  0,1\right]  $\textit{ and }$f\in C\left[
0,1\right]  .$
\end{enumerate}
\end{theorem}

\begin{corollary}
\label{thm1} Suppose that $T:C\left[  0,1\right]  \rightarrow C\left[
0,1\right]  $ is a positive linear operator such that
\begin{equation}
T\left(  e_{0}\right)  =e_{0},\ \ T\left(  e_{1}\right)  =e_{1}. \label{cc10}%
\end{equation}
Then there exists a linear positive operator $T^{\infty}:C\left[  0,1\right]
\rightarrow C\left[  0,1\right]  $ such that
\[
\lim_{m\rightarrow\infty}\left\Vert T^{\infty}\left(  f\right)  -T^{m}\left(
f\right)  \right\Vert =0,\ \ \ \mathit{\ }f\in C\left[  0,1\right]  .
\]
Furthermore the following pointwise estimate%
\begin{equation}
\left\vert T^{\infty}\left(  f;x\right)  -T^{m}\left(  f;x\right)  \right\vert
\leq\frac{3}{4}\omega_{2}\left(  f;\sqrt{\left\vert \left(  T^{\infty}%
-T^{m}\right)  \left(  e_{2};x\right)  \right\vert }\right)  \label{rr1}%
\end{equation}
holds true \textit{for }$x\in\left[  0,1\right]  $ \textit{and }$f\in C\left[
0,1\right]  .$
\end{corollary}

The second result shows that under the conditons $T\left(  e_{0}\right)
=e_{0},\ T\left(  e_{1}\right)  =e_{1},$ $T^{\infty}\left(  e_{2}\right)
=e_{1}$ the limit of the iteraes $T^{m}$ is exactly the operator $P\left(
f;x\right)  :=\left(  1-x\right)  f\left(  0\right)  +xf\left(  1\right)  ,$
and under the conditions $T\left(  e_{0}\right)  =e_{0},\ T\left(
e_{2}\right)  =e_{2},$ $T^{\infty}\left(  e_{1}\right)  =e_{2}$ the limit of
the iteraes $T^{m}$ is exactly the operator $V\left(  f;x\right)  :=\left(
1-x^{2}\right)  f\left(  0\right)  +x^{2}f\left(  1\right)  .$

\begin{theorem}
\label{t:11}Suppose that $T:C\left[  0,1\right]  \rightarrow C\left[
0,1\right]  $ is a positive linear operator.

\begin{enumerate}
\item[(a)] If $T\left(  e_{0}\right)  =e_{0},\ T\left(  e_{1}\right)
=e_{1},T^{\infty}\left(  e_{2}\right)  =e_{1},$ then $T^{\infty}=P$ and
\begin{equation}
\left\vert P\left(  f;x\right)  -T^{m}\left(  f;x\right)  \right\vert
\leq\frac{3}{4}\omega_{2}\left(  f;\sqrt{\left\vert x-T^{m}\left(
e_{2};x\right)  \right\vert }\right)  . \label{ee1}%
\end{equation}

\item[(b)] If $T\left(  e_{0}\right)  =e_{0},\ \ \ T\left(  e_{1}\right)  \leq
e_{1},\ \ \ T\left(  e_{2}\right)  =e_{2},T^{\infty}\left(  e_{1}\right)
=e_{2},$ then $T^{\infty}=V$ and
\begin{align*}
\left\vert V\left(  f;x\right)  -T^{m}\left(  f;x\right)  \right\vert  &
\leq3\omega_{2}\left(  f;\left\vert T^{m}\left(  e_{1};x\right)
-x^{2}\right\vert \right)  +2\omega_{1}\left(  f;\left\vert T^{m}\left(
e_{1};x\right)  -x^{2}\right\vert \right) \\
&  +\frac{3}{4}\omega_{2}\left(  f;\sqrt{2\left\vert T^{m}\left(
e_{1};x\right)  -x^{2}\right\vert }\right)  .
\end{align*}

\item[(c)] If $T\left(  e_{0}\right)  =e_{0},\ \ \ T\left(  e_{1}\right)  \geq
e_{1},\ \ \ T\left(  e_{2}\right)  =e_{2},T^{\infty}\left(  e_{1}\right)
=e_{2},$ then $T^{\infty}=V$ and
\[
\left\vert V\left(  f;x\right)  -T^{m}\left(  f;x\right)  \right\vert
\leq3\omega_{2}\left(  f;\left\vert T^{m}\left(  e_{1};x\right)
-x^{2}\right\vert \right)  +2\omega_{1}\left(  f;\left\vert T^{m}\left(
e_{1};x\right)  -x^{2}\right\vert \right)  .
\]

\end{enumerate}
\end{theorem}

\begin{remark}
(i) Results similar to that of Theorem \ref{t:11} (a) without the estimation
(\ref{ee1}) was obtained in \cite{rasa1} and \cite{gavrea2}. (ii) Theorem
\ref{t:11} (b) and \ref{t:11} (c) are new. They cover positive linear
operators which preserve $e_{0}$ and $e_{2}$. Convergence of overiterates of
Bernstein operators and discrete type positive linear operators preserving
$e_{0}$ and $e_{2}$ is studied in \cite{agr}, \cite{gonska6}.
\end{remark}

\begin{corollary}
\label{t:iterate}Let $T:C\left[  0,1\right]  \rightarrow C\left[  0,1\right]
$ be a positive linear operator such that%
\begin{equation}
T\left(  e_{0}\right)  =e_{0},\ \ T\left(  e_{1}\right)  =e_{1},\ T\left(
e_{2}\right)  \leq ae_{2}+be_{1},a,b\in R\backslash\left\{  0\right\}
,\ a+b=1. \label{cc11}%
\end{equation}
Then the pointwise approximation
\begin{equation}
\left\vert T^{m}\left(  f;x\right)  -P\left(  f;x\right)  \right\vert
\leq\frac{3}{4}\omega_{2}\left(  f;\sqrt{a^{m}x\left(  1-x\right)  }\right)
\label{est5}%
\end{equation}
holds true for all $x\in\left[  0,1\right]  $ and $f\in C\left[  0,1\right]  $.
\end{corollary}

\begin{remark}
It is worth mentioning that the conditions
\[
T\left(  e_{0}\right)  =e_{0},\ \ T\left(  e_{1}\right)  =e_{1},\ T\left(
e_{2}\right)  =ae_{2}+be_{1},a,b\in R\backslash\left\{  0\right\}  ,\ a+b=1.
\]
are satisfied by the many classical positive linear operators defined on
$C\left[  0,1\right]  $ and convergence of overiterates under these conditions
was studied in \cite{gonska3}, \cite{rasa1}, \cite{gavrea2}. The conditions
(\ref{cc11}) cover the classical MKZ and $q$-MKZ operators. The problem of
convergence of overiterates of MKZ operators without quantitative estimate was
studied in \cite{gavrea1}. In Corollary \ref{t:iterate} we give quantitative
estimate for convergence.
\end{remark}

\begin{corollary}
\label{c:11}Suppose that $T:C\left[  0,1\right]  \rightarrow C\left[
0,1\right]  $ is a positive linear operator.

\begin{enumerate}
\item[(a)] If $T\left(  e_{0}\right)  =e_{0},\ \ \ T\left(  e_{1}\right)  \leq
e_{1},\ \ T^{\infty}\left(  e_{1}\right)  =T^{\infty}\left(  e_{2}\right)
=0,$ then $T^{\infty}f=f\left(  0\right)  $ and%
\begin{align*}
\left\vert f\left(  0\right)  -T^{m}\left(  f;x\right)  \right\vert  &
\leq3\omega_{2}\left(  f;\left\vert T^{m}\left(  e_{1};x\right)  \right\vert
\right)  +2\omega_{1}\left(  f;\left\vert T^{m}\left(  e_{1};x\right)
\right\vert \right)  \\
&  +\frac{3}{4}\omega_{2}\left(  f;\sqrt{\left\vert T^{m}\left(
e_{2};x\right)  \right\vert +2\left\vert T^{m}\left(  e_{1};x\right)
\right\vert }\right)  .
\end{align*}

\item[(b)] If $T\left(  e_{0}\right)  =e_{0},\ \ \ T\left(  e_{1}\right)  \geq
e_{1},\ \ T^{\infty}\left(  e_{1}\right)  =T^{\infty}\left(  e_{2}\right)
=1,$ then $T^{\infty}f=f\left(  1\right)  $ and
\begin{align*}
\left\vert f\left(  1\right)  -T^{m}\left(  f;x\right)  \right\vert  &
\leq3\omega_{2}\left(  f;\left\vert T^{m}\left(  e_{1};x\right)  -1\right\vert
\right)  +2\omega_{1}\left(  f;\left\vert T^{m}\left(  e_{1};x\right)
-1\right\vert \right)  \\
&  +\frac{3}{4}\omega_{2}\left(  f;\sqrt{\left\vert T^{m}\left(
e_{2};x\right)  -1\right\vert }\right)  .
\end{align*}

\end{enumerate}
\end{corollary}

\begin{remark}
Corollary \ref{c:11} covers the Berstein-Stancu operators $S_{n}^{\left\langle
\alpha,\beta,\gamma\right\rangle }$ for some values of $\alpha,\beta,\gamma,$
see \cite{gonska5}. 
\end{remark}

\section{Proofs of the main results}

\begin{proof}
[Proof of Theorem \ref{thm0}](b) For every convex $g\in C^{2}\left[
0,1\right]  ,$ we have
\begin{equation}
g\left(  t\right)  \geq g\left(  x\right)  +g^{\prime}\left(  x\right)
\left(  t-x\right)  . \label{b0}%
\end{equation}
It follows that for any nondecreasing convex $g\in C^{2}\left[  0,1\right]  $
\begin{align*}
T\left(  g;x\right)   &  \geq g\left(  x\right)  +g^{\prime}\left(  x\right)
\left(  T\left(  e_{1};x\right)  -x\right)  \geq g\left(  x\right)  ,\\
g\left(  x\right)   &  \leq T^{m}\left(  g;x\right)  \leq T^{m+1}\left(
g;x\right)  \leq\left\Vert g\right\Vert .
\end{align*}
In other words the sequence continuous functions $\left\{  T^{m}\left(
g;\cdot\right)  \right\}  $ is nondecreasing for any nondecreasing convex
function $g\in C^{2}\left[  0,1\right]  .$ By Dini's theorem there exists
$T^{\infty}\left(  g;\cdot\right)  $ such that $T^{m}\left(  g;\cdot\right)
\rightarrow T^{\infty}\left(  g;\cdot\right)  $ uniformly on $\left[
0,1\right]  ,$ for any nonincreasing convex function $g\in C^{2}\left[
0,1\right]  .$In particular,%
\begin{align*}
\lim_{m\rightarrow\infty}\left\Vert T^{m}\left(  e_{1}\right)  -T^{\infty
}\left(  e_{1}\right)  \right\Vert  &  =0,\\
\lim_{m\rightarrow\infty}\left\Vert T^{m}\left(  e_{2}\right)  -T^{\infty
}\left(  e_{2}\right)  \right\Vert  &  =0.
\end{align*}

Let $g\in C^{2}\left[  0,1\right]  $ be arbitrary. Introduce the following
auxiliary functions%
\[
g_{\pm}\left(  t\right)  =\frac{1}{2}\left\Vert g^{\prime\prime}\right\Vert
t^{2}+\left\Vert g^{\prime}\right\Vert t\pm g\left(  t\right)  .
\]
It is clear that
\[
g_{\pm}^{\prime}\left(  t\right)  =\left\Vert g^{\prime\prime}\right\Vert
t+\left\Vert g^{\prime}\right\Vert \pm g\left(  t\right)  \geq0,\ \ \ g_{\pm
}^{\prime\prime}\left(  t\right)  =\left\Vert g^{\prime\prime}\right\Vert \pm
g\left(  t\right)  \geq0.
\]
Therefore the functions $g_{\pm}\left(  t\right)  $ are nondecreasing convex
for both choices of the sign. We have%
\begin{align*}
0  &  \leq T^{m+p}\left(  g_{\pm};x\right)  -T^{m}\left(  g_{\pm};x\right)
=\frac{1}{2}\left\Vert g^{\prime\prime}\right\Vert \left(  T^{m+p}\left(
e_{2};x\right)  -T^{m}\left(  e_{2};x\right)  \right) \\
&  +\left\Vert g^{\prime}\right\Vert \left(  T^{m}\left(  e_{1};x\right)
-T^{m+p}\left(  e_{1};x\right)  \right)  \pm\left(  T^{m+p}\left(  g;x\right)
-T^{m}\left(  g;x\right)  \right)  .
\end{align*}
It follows that%
\begin{equation}
\left\vert T^{m+p}\left(  g;x\right)  -T^{m}\left(  g;x\right)  \right\vert
\leq\frac{1}{2}\left\Vert g^{\prime\prime}\right\Vert \left\vert
T^{m+p}\left(  e_{2};x\right)  -T^{m}\left(  e_{2};x\right)  \right\vert
+\left\Vert g^{\prime}\right\Vert \left\vert T^{m}\left(  e_{1};x\right)
-T^{m+p}\left(  e_{1};x\right)  \right\vert . \label{b1}%
\end{equation}
So $\left\{  T^{m}\left(  g;x\right)  \right\}  $ is a Cauchy sequence in
$C\left[  0,1\right]  .$ Since $C\left[  0,1\right]  $ is complete there is a
function $f^{\infty}$ such that%
\[
\lim_{m\rightarrow\infty}\left\Vert T^{m}\left(  g\right)  -f^{\infty
}\right\Vert =0.
\]
Although this limit has been obtained for $g\in C^{2}\left[  0,1\right]  $
only, it extends to all $f\in C\left[  0,1\right]  $ by the Banach--Steinhaus
theorem. Hence we find an operator $T^{\infty}:C\left[  0,1\right]
\rightarrow C\left[  0,1\right]  $ say, such that $T^{\infty}f:=\lim
_{m\rightarrow\infty}T^{m}\left(  f\right)  =f^{\infty},$ $f\in C\left[
0,1\right]  $. Clearly, this operator is linear and positive.

Taking the limit as $p\rightarrow\infty$ in (\ref{b1}) we have%
\[
\left\vert \left(  T^{\infty}-T^{m}\right)  \left(  g;x\right)  \right\vert
\leq\frac{1}{2}\left\Vert g^{\prime\prime}\right\Vert \left\vert \left(
T^{\infty}-T^{m}\right)  \left(  e_{2};x\right)  \right\vert +\left\Vert
g^{\prime}\right\Vert \left\vert \left(  T^{m}-T^{\infty}\right)  \left(
e_{1};x\right)  \right\vert .
\]

Let $f\in C\left[  0,1\right]  $. For $g\in C^{2}\left[  0,1\right]  $
arbitrarily chosen we have the following estimate%
\begin{gather}
\left\vert T^{\infty}\left(  f;x\right)  -T^{m}\left(  f;x\right)  \right\vert
\leq\left\vert \left(  T^{\infty}-T^{m}\right)  \left(  f-g;x\right)
\right\vert +\left\vert T^{\infty}\left(  g;x\right)  -T^{m}\left(
g;x\right)  \right\vert \nonumber\\
\leq2\left\Vert f-g\right\Vert +\frac{1}{2}\left\Vert g^{\prime\prime
}\right\Vert \left\vert \left(  T^{\infty}-T^{m}\right)  \left(
e_{2};x\right)  \right\vert +\left\Vert g^{\prime}\right\Vert \left\vert
\left(  T^{m}-T^{\infty}\right)  \left(  e_{1};x\right)  \right\vert
\nonumber\\
=2\left\Vert f-g\right\Vert +\left\Vert g^{\prime}\right\Vert \left\vert
\left(  T^{m}-T^{\infty}\right)  \left(  e_{1};x\right)  \right\vert +\frac
{1}{2}\left\Vert g^{\prime\prime}\right\Vert \left\vert \left(  T^{\infty
}-T^{m}\right)  \left(  e_{2};x\right)  \right\vert . \label{b2}%
\end{gather}
We substitute now $g:=B_{n}\left(  Z_{h}f\right)  \in C^{2}\left[  0,1\right]
,$ where $Z_{h}f$ is Zhuk's function. In (\ref{b2}) using the inequalities
(\ref{zhuk}) and letting $\varepsilon\rightarrow0$ we arrive at%
\begin{align*}
\left\vert T^{\infty}\left(  f;x\right)  -T^{m}\left(  f;x\right)
\right\vert  &  \leq\frac{3}{2}\omega_{2}\left(  f;h\right)  +\frac{1}%
{h}\left(  2\omega_{1}\left(  f;h\right)  +\frac{3}{2}\omega_{2}\left(
f;h\right)  \right)  \left\vert \left(  T^{m}-T^{\infty}\right)  \left(
e_{1};x\right)  \right\vert \\
&  +\frac{1}{\delta^{2}}\frac{3}{4}\omega_{2}\left(  f;\delta\right)
\left\vert \left(  T^{\infty}-T^{m}\right)  \left(  e_{2};x\right)
\right\vert
\end{align*}
with $h>0$ and $\delta>0$. If%
\[
\left\vert \left(  T^{m}-T^{\infty}\right)  \left(  e_{1};x\right)
\right\vert >0,\ \ \ \left\vert \left(  T^{m}-T^{\infty}\right)  \left(
e_{2};x\right)  \right\vert >0
\]
taking $h=\left\vert \left(  T^{m}-T^{\infty}\right)  \left(  e_{1};x\right)
\right\vert $ and $\delta^{2}=\left\vert \left(  T^{\infty}-T^{m}\right)
\left(  e_{2};x\right)  \right\vert $ yields the desired result. If
$\left\vert \left(  T^{m}-T^{\infty}\right)  \left(  e_{1};x\right)
\right\vert =0$ and $\left\vert \left(  T^{\infty}-T^{m}\right)  \left(
e_{2};x\right)  \right\vert >0$, then
\[
\left\vert T^{\infty}\left(  f;x\right)  -T^{m}\left(  f;x\right)  \right\vert
\leq\frac{3}{2}\omega_{2}\left(  f;h\right)  +\frac{1}{\delta^{2}}\frac{3}%
{4}\omega_{2}\left(  f;\delta\right)  \left\vert \left(  T^{m}-T^{\infty
}\right)  \left(  e_{2};x\right)  \right\vert
\]
for all $h>0.$ Taking $\delta=\left\vert \left(  T^{\infty}-T^{m}\right)
\left(  e_{2};x\right)  \right\vert ,$ $h\rightarrow0$ yields the desired
result. If $\left\vert \left(  T^{m}-T^{\infty}\right)  \left(  e_{1}%
;x\right)  \right\vert >0$ and $\left\vert \left(  T^{\infty}-T^{m}\right)
\left(  e_{2};x\right)  \right\vert =0$, then
\[
\left\vert T^{\infty}\left(  f;x\right)  -T^{m}\left(  f;x\right)  \right\vert
\leq\frac{3}{2}\omega_{2}\left(  f;h\right)  +\frac{1}{h}\left(  2\omega
_{1}\left(  f;h\right)  +\frac{3}{2}\omega_{2}\left(  f;h\right)  \right)
\left\vert \left(  T^{m}-T^{\infty}\right)  \left(  e_{1};x\right)
\right\vert
\]
for all $h>0.$ Taking $h=\left\vert \left(  T^{\infty}-T^{m}\right)  \left(
e_{1};x\right)  \right\vert ,$ yields the desired result. If $\left(
T^{m}-T^{\infty}\right)  \left(  e_{1};x\right)  =0$ and $\left(
T^{m}-T^{\infty}\right)  \left(  e_{2};x\right)  =0$, then
\[
\left\vert T^{\infty}\left(  f;x\right)  -T^{m}\left(  f;x\right)  \right\vert
\leq\frac{3}{2}\omega_{2}\left(  f;h\right)  .
\]
For $h\rightarrow0$ we obtain $T^{\infty}\left(  f;x\right)  =T^{m}\left(
f;x\right)  $ for all $0\leq x\leq1.$

(a) It follows that from (\ref{b0}) that any nonincreasing convex $g\in
C^{2}\left[  0,1\right]  $
\begin{align*}
T\left(  g;x\right)   &  \geq g\left(  x\right)  +g^{\prime}\left(  x\right)
\left(  T\left(  e_{1};x\right)  -x\right)  \geq g\left(  x\right)  ,\\
g\left(  x\right)   &  \leq T^{m}\left(  g;x\right)  \leq T^{m+1}\left(
g;x\right)  \leq\left\Vert g\right\Vert .
\end{align*}
In other words the sequence continuous functions $\left\{  T^{m}\left(
g;\cdot\right)  \right\}  $ is nondecreasing for any nonincreasing convex
function $g\in C^{2}\left[  0,1\right]  .$ By Dini's theorem there exists
$T^{\infty}\left(  g;\cdot\right)  $ such that $T^{m}\left(  g;\cdot\right)
\rightarrow T^{\infty}\left(  g;\cdot\right)  $ uniformly on $\left[
0,1\right]  ,$ for any nonincreasing convex function $g\in C^{2}\left[
0,1\right]  .$In particular,%
\begin{align*}
\lim_{m\rightarrow\infty}\left\Vert T^{m}\left(  -e_{1}\right)  -T^{\infty
}\left(  -e_{1}\right)  \right\Vert  &  =0,\\
\lim_{m\rightarrow\infty}\left\Vert T^{m}\left(  \left(  e_{0}-e_{1}\right)
^{2}\right)  -T^{\infty}\left(  \left(  e_{0}-e_{1}\right)  ^{2}\right)
\right\Vert  &  =0,
\end{align*}
and $\lim_{m\rightarrow\infty}\left\Vert T^{m}\left(  e_{2}\right)
-T^{\infty}\left(  e_{2}\right)  \right\Vert =0$ since $T^{m}\left(  \left(
e_{0}-e_{1}\right)  ^{2}\right)  =T^{m}\left(  e_{0}\right)  -2T^{m}\left(
e_{1}\right)  +T^{m}\left(  e_{2}\right)  $

Let $g\in C^{2}\left[  0,1\right]  $ be arbitrary. Introduce the following
auxiliary functions%
\[
g_{\pm}\left(  t\right)  =\frac{1}{2}\left\Vert g^{\prime\prime}\right\Vert
\left(  1-t\right)  ^{2}+\left\Vert g^{\prime}\right\Vert \left(  1-t\right)
\pm g\left(  t\right)  .
\]
It is clear that
\[
g_{\pm}^{\prime}\left(  t\right)  =-\left\Vert g^{\prime\prime}\right\Vert
\left(  1-t\right)  -\left\Vert g^{\prime}\right\Vert \pm g\left(  t\right)
\leq0,\ \ \ g_{\pm}^{\prime\prime}\left(  t\right)  =\left\Vert g^{\prime
\prime}\right\Vert \pm g\left(  t\right)  \geq0.
\]
Therefore the functions $g_{\pm}\left(  t\right)  $ are nonincreasing convex
for both choices of the sign. We have%
\begin{align*}
0 &  \leq T^{m+p}\left(  g_{\pm};x\right)  -T^{m}\left(  g_{\pm};x\right)
=\frac{1}{2}\left\Vert g^{\prime\prime}\right\Vert \left(  T^{m+p}\left(
e_{2};x\right)  -T^{m}\left(  e_{2};x\right)  \right)  \\
&  +\left(  \left\Vert g^{\prime\prime}\right\Vert +\left\Vert g^{\prime
}\right\Vert \right)  \left(  T^{m}\left(  e_{1};x\right)  -T^{m+p}\left(
e_{1};x\right)  \right)  \pm\left(  T^{m+p}\left(  g;x\right)  -T^{m}\left(
g;x\right)  \right)  .
\end{align*}
The rest of the proof is similar to that of part (b).
\end{proof}

\begin{proof}
[Proof of Theorem \ref{t:11}](a) It remains to show that $T^{\infty}\left(
f\right)  =P\left(  f\right)  $ for all $f\in C\left[  0,1\right]  $. It is
clear that it is enough to show this equality in $C^{2}\left[  0,1\right]  $.
Let $g\in C^{2}\left[  0,1\right]  $. Define the following auxiliary
functions.%
\[
G\left(  x\right)  :=g\left(  x\right)  -P\left(  g;x\right)  ,\ \ \ l:=\frac
{1}{2}\left\Vert G^{\prime\prime}\right\Vert =\frac{1}{2}\left\Vert
g^{\prime\prime}\right\Vert ,\ \ g_{\pm}\left(  x\right)  :=-lx^{2}+lx\pm
G\left(  x\right)  .
\]
It is clear that $g_{\pm}$ is concave and nonnegative, since
\[
g_{\pm}^{\prime\prime}\left(  x\right)  =-\left\Vert G^{\prime\prime
}\right\Vert \pm G^{\prime\prime}\left(  x\right)  \leq0,\ \ \ G\left(
0\right)  =G\left(  1\right)  =0.
\]
It follows that
\[
-l\left(  x-x^{2}\right)  \leq G\left(  x\right)  \leq l\left(  x-x^{2}%
\right)  ,\ \ \ \ 0\leq x\leq1.
\]
Applying the positive operator $T^{\infty}$ we get%
\[
-l\left(  T^{\infty}\left(  e_{1};x\right)  -T^{\infty}\left(  e_{2};x\right)
\right)  \leq T^{\infty}\left(  G;x\right)  =T^{\infty}\left(  g;x\right)
-P\left(  g;x\right)  \leq l\left(  T^{\infty}\left(  e_{1};x\right)
-T^{\infty}\left(  e_{2};x\right)  \right)  ,\ \ \ \
\]
for all $0\leq x\leq1,$ and consequently%
\[
T^{\infty}\left(  g\right)  =P\left(  g\right)
\]
for all $g\in C^{2}\left[  0,1\right]  $, which completes the proof.

(b) The operator $T^{\infty}$ of Theorem \ref{thm1} satisfies%
\[
T^{\infty}\left(  e_{0}\right)  =e_{0},\ \ T^{\infty}\left(  e_{1}\right)
=e_{2},\ \ \ T^{\infty}\left(  e_{2}\right)  =e_{2}.
\]
It remains to show that $T^{\infty}\left(  f\right)  =V\left(  f\right)  $ for
all $f\in C\left[  0,1\right]  $. It is clear that it is enough to show this
equality in $C^{2}\left[  0,1\right]  $.

Let $g\in C^{2}\left[  0,1\right]  $. Define the following auxiliary
functions.%
\begin{align*}
G\left(  x\right)   &  :=g\left(  x\right)  -V\left(  g;x\right)  =g\left(
x\right)  -\left(  1-x^{2}\right)  g\left(  0\right)  -x^{2}g\left(  1\right)
,\ \ \ \\
l &  :=\frac{1}{2}\left\Vert g^{\prime\prime}+2g\left(  0\right)  -2g\left(
1\right)  \right\Vert ,\ \ G^{\prime\prime}\left(  x\right)  =g^{\prime\prime
}\left(  x\right)  +2g\left(  0\right)  -2g\left(  1\right)  ,\ \\
g_{\pm}\left(  x\right)   &  =-lx^{2}+lx\pm G\left(  x\right)  .
\end{align*}
It is clear that $g_{\pm}$ is concave and nonnegative, since%
\[
g_{\pm}^{\prime\prime}\left(  x\right)  =-\left\Vert G^{\prime\prime
}\right\Vert \pm G^{\prime\prime}\left(  x\right)  \leq0,\ \ \ G\left(
0\right)  =G\left(  1\right)  =0.
\]
It follows that
\[
-l\left(  x-x^{2}\right)  \leq G\left(  x\right)  \leq l\left(  x-x^{2}%
\right)  ,\ \ \ \ 0\leq x\leq1.
\]
Application of the positive operator $T^{\infty}$ gives%
\[
-l\left(  T^{\infty}\left(  e_{1};x\right)  -T^{\infty}\left(  e_{2};x\right)
\right)  \leq T^{\infty}\left(  G;x\right)  =T^{\infty}\left(  g;x\right)
-V\left(  g;x\right)  \leq l\left(  T^{\infty}\left(  e_{1};x\right)
-T^{\infty}\left(  e_{2};x\right)  \right)  ,\ \ \ \ 0\leq x\leq1,
\]
and $T^{\infty}\left(  g\right)  =V\left(  g\right)  $ for all $g\in
C^{2}\left[  0,1\right]  .$
\end{proof}

\begin{proof}
[Proof of Corollary \ref{t:iterate}]By the induction we have
\[
x^{2}\leq T^{m}\left(  e_{2};x\right)  \leq a^{m}x^{2}+b\left(
1+a+...+a^{m-1}\right)  x=a^{m}x^{2}+\left(  1-a^{m}\right)  x.
\]
So
\[
0\leq x-T^{m}\left(  e_{2};x\right)  \leq a^{m}x\left(  1-x\right)  .
\]

\end{proof}

\begin{proof}
[Proof of Corollary \ref{c:11}] The proof is based on the Taylor formula about
the point $0$ for part (a) and about the point $1$ for part $(b).$
\end{proof}

\end{document}